\documentclass[11pt]{article}
\usepackage[dvips]{epsfig}
\usepackage{amsmath,amssymb,euscript,graphicx,amsfonts,verbatim}
\usepackage{enumerate,color}
\usepackage{graphicx}
\usepackage{enumitem}
\usepackage{enumerate,color}
\usepackage{graphicx}

\usepackage[colorlinks=true,linkcolor=blue,urlcolor=blue,citecolor=blue]{hyperref}

\newtheorem{theorem}{Theorem}[section]
\newtheorem{proposition}[theorem]{Proposition}
\newtheorem{lemma}[theorem]{Lemma}
\newtheorem{corollary}[theorem]{Corollary}

\newenvironment{proof}{{\noindent \sc Proof.}}{\hfill $\Qed$\\}

\newcommand{\Qed}{\rule{2.5mm}{3mm}}

\newcommand{\G}{\Gamma}

\newcommand{\QQ}{\mathbb{Q}}
\newcommand{\RR}{\mathbb{R}}
\newcommand{\CC}{\mathbb{C}}

\DeclareMathOperator{\mat}{Mat}

\newcommand{\MX}{\mat_X(\CC)}

\newcounter{case}

\renewcommand{\thecase}{\arabic{case}}

\newcounter{subcase}

\numberwithin{subcase}{case}

\title{A CLASSIFICATION OF $Q$-POLYNOMIAL DISTANCE-REGULAR GRAPHS WITH GIRTH $6$}

\author{\v{S}tefko Miklavi\v{c} 
	\thanks{This work is supported in part by the Slovenian Research and Innovation Agency (research program P1-0285 and research projects J1-3001, J1-3003, J1-4008, J1-4084, N1-0208, N1-0353, J1-50000, J1-60012).
	} \\ [3mm]
	Andrej Maru\v si\v c Institute,  University of Primorska \\
	Muzejski trg 2, 6000 Koper, Slovenia, and \\[3mm]
	Institute of Mathematics, Physics and Mechanics \\
	Jadranska 19, 1000 Ljubljana, Slovenia \\[3mm]
	stefko.miklavic@upr.si}
\begin{document}
 \maketitle
 
 \begin{abstract}
Let $\G$ denote a $Q$-polynomial distance-regular graph with diameter $D$ and valency $k \ge 3$. In [Homotopy in $Q$-polynomial distance-regular graphs, Discrete Math., {\bf 223} (2000), 189–206], H. Lewis  showed that the girth of $\G$ is at most $6$. In this paper we classify graphs that attain this upper bound. We show that $\G$ has girth $6$ if and only if it is either isomorphic to the Odd graph on a set of cardinality $2D +1$, or to a generalized hexagon of order $(1, k -1)$. 
 \end{abstract}

\section{Introduction} \label{sec:intro}

Let $\G$ denote a $Q$-polynomial distance-regular graph with diameter $D$, valency $k\ge 3$,
and intersection numbers $a_i, b_i, c_i$ (see Sections 2 and 3 for formal definitions). In \cite[Corollary 30]{Le}, H. Lewis showed that the girth of $\G$ is at most $6$. Therefore, the problem of classification of $Q$-polynomial distance-regular graphs that attain this upper bound arises naturally. In the present paper we show that $\G$ has girth $6$ if and only if it is either isomorphic to the Odd graph on a set of cardinality $2D + 1$, or to a generalized hexagon of order $(1, k -1)$. Our strategy for proving the above result is as follows. If $\G$ has girth $6$, then $a_1 = a_2 = 0$ and $c_2 = 1$. It follows from \cite[Theorem 6.3]{Mi1} that in this case $\G$ is either bipartite or almost bipartite. We treat these two cases separately. We now briefly summarize what was done so far for each of the above two cases.

The classification of almost bipartite $Q$-polynomial distance-regular graphs with $D \ge 4$ is given in \cite[Theorem 1.1]{LT}. It transpires that if $\G$ is almost bipartite with $D \ge 4$ and girth $6$, then it is the Odd graph on a set of cardinality $2D+1$. Therefore, it remains to consider the case $D=3$. We do this in Section \ref{sec:almostbip}.

Assume next that $\G$ is bipartite with diameter $D \ge 3$ and valency $k \ge 3$.  In \cite[Theorem 1.1]{Ca1}, Caughman showed that if $D \ge 12$, then $\G$ is the $D$-dimensional hypercube, or the folded $2D$-dimensional hypercube, or the intersection numbers of $\G$ satisfy $c_i = (q^i - 1)/(q - 1) \;  (1 \le i \le D)$, where $q \ge 2$ is an integer. In \cite{Mi3}, this result was extended to the case $D \ge 9$. Therefore, if $D \ge 9$, then $c_2 \ge 2$ (and so $\G$ has girth $4$). In \cite[Theorem 6.1]{Mi2} we proved that if $D=4$, then again $c_2 \ge 2$, hence also in this case girth equals $4$. If $D=3$, then $\G$ is clearly isomorphic to a generalized hexagon of order $(1,k-1)$. In the present paper we focus on the case $5 \le D \le 8$. It turns out that the case $D=5$ is the most difficult one. Therefore, we treate the cases $D \ge 6$ and $D=5$ in separate sections. In Section \ref{sec:bipDne5}, we show that if $D \ge 6$, then $c_2\ge 2$, and so $\G$ has girth 4. We show the same result for the case $D=5$ in Section \ref{sec:bipDeq5}.
\section{Preliminaries} \label{sec:prelim}

In this section, we review some definitions and basic concepts regarding distance-regular graphs. See the book of Brouwer, Cohen and Neumaier \cite{BCN} for more background information.

Throughout this paper, $\G = (X,R)$ will denote a finite, undirected, connected graph,
without loops or multiple edges, with vertex set $X$, edge set $R$, path length distance
function $\partial$, and diameter $D := max\{\partial(x, y)|x, y \in X\}$. For a vertex $x \in X$ define $\G_i(x)$ to be the set of vertices at distance $i$ from $x$. We abbreviate $\G(x) := \G_1(x)$. Let $k$ denote a non-negative integer. Then $\G$ is said to be {\em regular with valency} $k$ whenever $|\G(x)| = k$ for all $x \in X$. The graph $\G$ is said to be {\em distance-regular} whenever for all integers $h, i, j \;
(0 \le h, i, j \le D)$, and all $x, y \in X$ with $\partial(x, y) = h$, the number
$$
p^h_{ij} := |\G_i(x) \cap \G_j(y)|
$$
is independent of $x, y$. The constants $p^h_{ij}$ are known as the {\em intersection numbers} of $\G$.
For convenience, set $c_i := p^i_{1,i-1}$ for $1 \le i \le D$, $a_i := p^i_{1i}$ for $0 \le i \le D$, $b_i := p^i_{1,i+1}$ for $0 \le i \le D -1$, $k_i := p^0_{ii}$ for $0 \le i \le D$, and $c_0 = b_D = 0$. We observe $a_0 = 0$ and $c_1 = 1$. Moreover, $\G$ is regular with valency $k = b_0=k_1$, and $c_i + a_i + b_i = k$ for $0 \le i \le D$. It is well-known that 
\begin{equation}
	\label{eq:ki}
  k_i = \frac{b_0 b_1 \cdots b_{i-1}}{c_1 c_2 \cdots c_i}
\end{equation}
for $0 \le i \le D$. Observe that $\G$ is bipartite if and only if $a_i = 0$ for $0 \le i \le D$. In this case $b_i +c_i = k$ for $0 \le i \le D$. We say that $\G$ is {\em almost bipartite}, whenever $a_i = 0$ for $0 \le i \le D - 1$ and $a_D \ne 0$. The following families of distance-regular graphs will appear in this paper: hypercubes (see \cite[Section 9.2]{BCN}), folded hypercubes (see \cite[Subsection 9.2 D]{BCN}), Odd graphs (see \cite[Subsection 9.1 D]{BCN}), and generalized hexagons (see \cite[Section 6.5]{BCN}). Observe that if $D \le 2$, then the girth of $\G$ is at most $5$. Therefore, from now on we assume $\G$ is distance-regular with diameter $D \ge 3$ and valency $k \ge 3$.v Throughout the paper we will be extensively using the following result without explicitly referring to it. 

\begin{lemma}
	\label{lem:girth6}
	Let $\G$ denote a distance-regular graph with diameter $D \ge 3$ and valency $k \ge 3$. Then the following (i), (ii) hold.
	\begin{itemize}
		\item[(i)] The girth of $\G$ is at least $6$ if and only if $a_1=a_2=0$ and $c_2=1$.
		\item[(ii)] If $\G$ is isomorphic to the $D$-cube, or to the folded $2D$-cube, or to the folded $(2D+1)$-cube, then $a_1=0$ and $c_2=2$. In particular, the girth of $\G$ is equal to $4$.
	\end{itemize}
\end{lemma}

We recall the Bose-Mesner algebra of $\G$. Let $\MX$ denote the $\CC$-algebra consisting of the matrices over $\CC$ which have rows and columns indexed by $X$. For $0 \le i \le D$ let $A_i$ denote the matrix in $\MX$ with $x, y$ entry
$$
\left( A_i\right) _{xy}=
\begin{cases}
	\hspace{0.2cm} 1 \hspace{0.5cm} \text{if} & \partial(x,y)=i,   \\
	\hspace{0.2cm} 0 \hspace{0.5cm} \text{if} &  \partial(x,y) \neq i
\end{cases} \qquad (x,y \in X).
$$
We call $A_i$ the $i$-th {\em distance matrix of} $\G$. We abbreviate $A := A_1$ and call $A$ the {\em adjacency matrix of} $\G$. The matrices $A_0, A_1, \ldots, A_D$ form a basis for a commutative semi-simple
$\CC$-algebra $M$, known as the {\em Bose-Mesner algebra}, see for example \cite[p. 44]{BCN}. By \cite[Theorem 2.6.1]{BCN} the algebra $M$ has a second basis $E_0, E_1, \ldots,E_D$, such that $E_i E_j=\delta_{ij} E_i \; (0 \le i, j \le D)$, $E_0 + E_1 + \cdots + E_D=I$, and $E_0 = |X|^{-1} J$, where $I$ and $J$ denote the identity and the all 1's matrix of $\MX$, respectively. The $E_0, E_1, \ldots, E_D$ are known as the {\em primitive idempotents of} $\G$, and we refer to $E_0$ as the {\em trivial idempotent} of $\G$. 

For $0 \le i \le D$ define a real number $\theta_i$ by $A = \sum_{i=0}^D \theta_i E_i$. Then $AE_i = E_iA = \theta_i E_i$ for $0 \le i \le D$. The scalars $\theta_0, \theta_1, \ldots, \theta_D$ are distinct since $A$ generates $M$ \cite[p. 197]{BI}. The scalars $\theta_0, \theta_1, \ldots, \theta_D$ are known as the {\em eigenvalues of} $\G$. We remark $k \ge \theta_i \ge -k$ for $0 \le i \le D$, and $\theta_0=k$ \cite[p. 45]{BCN}.


\section{The $Q$-polynomial property}
\label{sec:Q-poly}

In this section, we recall the $Q$-polynomial property of distance-regular graphs.
Let $\Gamma$ denote a distance-regular graph with diameter $D \ge 3$ and valency $k \ge 3$, and let $A_0, A_1, \ldots, A_D$ denote the distance matrices of $\G$. Observe that $A_i \circ A_j = \delta_{ij}A_i \; (0 \le i,j \le D)$, where $\circ$ denotes the entrywise multiplication, and so the algebra $M$ is closed under $\circ$. Let $E_0, E_1, \ldots, E_D$ denote the primitive idempotents of $\G$.
The {\it Krein parameters} $q_{ij}^h \; (0 \le h,i,j \le D)$ of $\Gamma$ are 
defined by
\begin{equation}
	E_i \circ E_j = |X|^{-1} \sum_{h=0}^D q_{ij}^h E_h \;\;\; 
	(0 \le i,j \le D).
\end{equation}
We say $\Gamma$ is 
{\it $Q$-polynomial} (with respect to the given ordering 
$E_0, E_1, \ldots, E_D$ of the primitive idempotents or with respect to the given ordering 
$\theta_0, \theta_1, \ldots, \theta_D$ of the corresponding eigenvalues), whenever for all 
distinct integers $i,j \; (0 \le i,j \le D)$ the following holds:
$q_{ij}^1 \ne 0 \; \hbox{ if and only if } \; |i-j|=1$. The following result motivated the research presented in this paper.

\begin{theorem}
	\label{thm:lewis}
	{\rm (\cite[Corollary 30]{Le})} Let $\Gamma$ denote a $Q$-polynomial distance-regular graph with diameter 
	$D \ge 3$ and valency $k \ge 3$. Then the girth of $\G$ is at most $6$.
\end{theorem}

\noindent
The following theorem and its corollary will be crucial in this paper.
\begin{theorem}
	\label{thm:girth6}
	{\rm (\cite[Theorem 6.3]{Mi1})} Let $\Gamma$ denote a $Q$-polynomial distance-regular graph with diameter $D \ge 3$, valency $k \ge 3$, and intersection number $a_1=0$. Then exactly one of the following (i)-(iii) holds.
	\begin{itemize}
		\item[(i)] $\Gamma$ is bipartite;
		\item[(ii)] $\Gamma$ is almost bipartite;
		\item[(iii)] $a_i \ne 0$ for $2 \le i \le D$.
	\end{itemize}
\end{theorem}

\begin{corollary}
	\label{cor:girth6}
	Let $\Gamma$ denote a $Q$-polynomial distance-regular graph with diameter $D \ge 3$, valency $k \ge 3$, and girth $6$. Then $\G$ is either bipartite or almost bipartite.
\end{corollary}

\begin{proof}
	Recall that by Lemma \ref{lem:girth6} we have $a_1=a_2=0$. The result now follows from Theorem \ref{thm:girth6}.
\end{proof}

\smallskip \noindent
We now record several important results about $Q$-polynomial distance-regular graphs that we will use later in the paper. 

\begin{theorem}
	\label{thm:almbipLT}
	{\rm (\cite[Theorem 1.1, Remark 1.4]{LT})}
	Let $\G$ denote an almost bipartite distance-regular graph with diameter $D \ge 3$ and valency $k \ge 3$. Then $\G$ is $Q$-polynomial if and only if at least one of (i)--(iii) below holds.
	\begin{itemize}
		\item[(i)] $\G$ is the folded $(2D + 1)$-cube.
		\item[(ii)] $\G$ is the Odd graph on a set of cardinality $2D + 1$.
		\item[(iii)] $D = 3$ and there exist a complex scalar $\beta$, such that the intersection numbers of $\G$ satisfy $k = 1 + (\beta^2 - 1)(\beta(\beta + 2) - (\beta + 1) c_2)$ and  $c_3 = -(\beta + 1)(\beta^2 + \beta - 1 - (\beta + 1) c_2)$.
	\end{itemize}
    Suppose that (iii) holds but none of (i),(ii) do. Then $\beta$ is unique, integral and less than $-2$.
\end{theorem}

\begin{lemma}
	\label{lem:caug}
	{\rm (\cite[Lemma 3.2, Lemma 3.3]{Ca1})} Let $\Gamma$ denote a bipartite distance-regular graph with diameter $D \ge 4$, valency 
	$k \ge 3$, and intersection numbers $b_i, c_i$ $(0 \le i \le D)$. 
	We assume $\Gamma$ is $Q$-polynomial with respect to 
	$E_0, E_1, \ldots, E_D$. For $0 \le i \le D$ let $\theta_i$ denote the eigenvalue associated 
	with $E_i$. Assume $\Gamma$ is not the $D$-cube or 
	the folded $2D$-cube. Then there exist scalars $q,s^* \in \RR$ such
	that (i)--(iii) hold.
	\begin{itemize}
		\item[(i)]   $|q| > 1, \;$ $s^* q^i \ne 1 \qquad (2 \le i \le 2D+1)$;
		\item[(ii)]  $\theta_i=h(q^{D-i}-q^i)$
		for $0 \le i \le D$, where 
		$$
		h=\frac{1-s^* q^3}{(q-1)(1-s^* q^{D+2})}.
		$$
		\item[(iii)] $k=c_D=h(q^D-1)$, and
		$$
		c_i=\frac{h(q^i-1)(1-s^* q^{D+i+1})}{1-s^* q^{2i+1}}, \qquad 
		b_i=\frac{h(q^D-q^i)(1-s^* q^{i+1})}{1-s^* q^{2i+1}} \qquad (1 \le i \le D-1).
		$$
	\end{itemize}
\end{lemma}

In the proposition below we collect some results about the scalars $q$ and $s^*$ from Lemma \ref{lem:caug} that we will need later in the paper. 

\begin{proposition}
	\label{prop:sandq}
	{\rm (\cite[Theorem 4.1, Lemma 5.1]{Ca1} , \cite[Lemma 3.2, Lemma 3.3]{Ca1a})}
	Let $\Gamma$ denote a bipartite distance-regular graph with diameter $D \ge 4$, valency 
	$k \ge 3$, and intersection numbers $b_i, c_i$ $(0 \le i \le D)$. 
	We assume $\Gamma$ is $Q$-polynomial with respect to 
	$E_0, E_1, \ldots, E_D$. For $0 \le i \le D$ let $\theta_i$ denote the eigenvalue associated 
	with $E_i$. Assume $\Gamma$ is not the $D$-cube or the folded $2D$-cube. Let $q, s^*$ be scalars as in Lemma \ref{lem:caug}. Define $\beta=q+1/q$. Then the following (i)--(iv) hold.
	\begin{itemize}
		\item[(i)] If $D \ge 6$, then $q > 1$.
		\item[(ii)] If $q > 1$, then $-q^{-D-1} \le s^* < q^{-2D-1}$.
		\item[(iii)] $\theta_1 \ne -1$ and 
		$$
		  \beta=\frac{\theta_1^2+c_2 \theta_1+b_2 (k-2)}{b_2(\theta_1+1)}.
		$$
		\item[(iv)] If $D \ge 5$, then $\theta_0, \theta_1, \ldots, \theta_D$ are integers.
	\end{itemize}
In particular, if $D \ge 5$, then $\beta \in \QQ$.
\end{proposition}

\section{$\G$ is almost bipartite}
\label{sec:almostbip}

In this section we consider the case where $\G$ is an almost bipartite $Q$-polynomial distance-regular graph with diameter $D \ge 3$, valency $k \ge 3$, and girth $6$. Our main result is the following theorem.

\begin{theorem}
	\label{thm:almost}
	Let $\G$ denote an almost bipartite $Q$-polynomial distance-regular graph with diameter $D \ge 3$, valency $k \ge 3$, and girth $6$.  Then $\G$ is  the Odd graph on a set with cardinality $2D+1$.
\end{theorem}

\begin{proof}
	Consider  first the case $D=3$. We assume that $\G$ is not the Odd graph on a set with cardinality $7$ and obtain a contradiction. Recall that $\G$ is also not the folded $7$-cube, as the folded $7$-cube has intersection number $c_2=2$, and so its girth equals $4$. As $c_2=1$, we therefore obtain by Theorem \ref{thm:almbipLT} that the intersection numbers of $\G$ are given by $b_0 = 1 + (\beta^2 - 1)(\beta^2 + \beta - 1)$ and $c_3=-(\beta+1)(\beta^2 - 2)$, where $\beta$ is an integer less than $-2$. By \eqref{eq:ki} we find
	$$
	k_3 =\frac{b_0 b_1 b_2}{c_1 c_2 c_3} = \frac{b_0(b_0 - 1)^2}{c_3} =-1-5\beta+7\beta^3-6\beta^4-7\beta^5+5\beta^6+3\beta^7-2\beta^8-\beta^9-\frac{3\beta+4}{\beta^2-2},
	$$
	implying that $(3\beta+4)/(\beta^2-2)$ is an integer. However, this is a contradiction, since $\beta^2-2 > |3\beta+4|$ as $\beta \le -3$. This shows that $\G$ is the Odd graph on a set of cardinality $7$. 
	
	Next we consider the case $D \ge 4$. By Theorem \ref{thm:almbipLT}, $\G$ is either the folded $(2D+1)$-cube, or the Odd graph on a set with cardinality $2D+1$. Recall that the girth of the folded $(2D+1)$-cube is $4$, and so the result follows.
\end{proof}


\section{$\G$ is bipartite: case $D \ne 5$}
\label{sec:bipDne5}

In this section, we consider the case where $\G$ is a bipartite $Q$-polynomial distance-regular graph
with diameter $D \ne 5$ and valency $k \ge 3$. Our main result is that if $\G$ has girth $6$, then $D=3$ and $\G$ is isomorphic to  a generalized hexagon of order $(1, k - 1)$. We split our argument in a series of propositions. 

\begin{proposition}
	\label{prop:D=3}
	Let $\G$ denote a bipartite $Q$-polynomial distance-regular graph with diameter $D = 3$, valency $k \ge 3$, and girth $6$. Then $\G$ is isomorphic to  a generalized hexagon of order $(1, k - 1)$.
\end{proposition}

\begin{proof}
	Since $c_2=1$, the intersection numbers of $\G$ satisfy $k=b_0=c_3=b_1+1=b_2+1$. It follows that $\G$ is a generalized hexagon of order $(1, k - 1)$, see \cite[Subsection 6.5]{BCN}.
\end{proof}

\begin{proposition}
	\label{prop:D=4}
	{\rm (\cite[Theorem 6.1]{Mi2})}  There are no bipartite $Q$-polynomial distance-regular graphs with diameter $D = 4$, valency $k \ge 3$, and girth $6$. 
\end{proposition}

We now consider the case $D \ge 6$ (although some of the results in the reminder of this section hold for $D=5$ as well). We show that in such a graph we have $c_2 \ge 2$ (and consequently the girth of $\G$ is equal to $4$). As already mentioned in the Introduction, if $D \ge 9$ this follows from \cite{Ca1} and \cite{Mi3}, but for the convenience of the reader we include the case $D \ge 9$ in our treatment below. 

\begin{lemma}
	\label{lem:c2=1}
	Let $\G$ denote a bipartite $Q$-polynomial distance-regular graph with diameter $D \ge 5$ and valency $k \ge 3$, which is not the $D$-cube or the folded $2D$-cube. Let $s^*,q$ be scalars as in Lemma \ref{lem:caug}. Abbreviate $\alpha=1 + q - q^2 - q^{D-1} + q^D + q^{D+1}$. Then $c_2=1$ if and only if
	$$
	s^* = \frac{\alpha \pm \sqrt{\alpha^2 - 4 q^{D+1}}}{2q^{D+3}}.
	$$
\end{lemma}

\begin{proof}
	This follows immediately from Lemma  \ref{lem:caug}(iii).
\end{proof}

\begin{proposition}
	\label{prop:q>1}
	Let $\G$ denote a bipartite $Q$-polynomial distance-regular graph with diameter $D \ge 5$ and valency $k \ge 3$, which is not the $D$-cube or the folded $2D$-cube. Let $s^*,q$ be scalars as in Lemma \ref{lem:caug}, and recall that $|q| > 1$. If $q > 1$, then $c_2 \ge 2$.
\end{proposition}

\begin{proof}
Recall the scalar $\alpha=1 + q - q^2 - q^{D-1} + q^D + q^{D+1}$  from Lemma \ref{lem:c2=1} and observe that $\alpha \ge 2$ as $q > 1$. We claim 
$$
\frac{\alpha - \sqrt{\alpha^2 - 4 q^{D+1}}}{2q^{D+3}} > q^{-2D-1}.
$$
First observe that $(\alpha q^{D-2}-2)^2-q^{2D-4}(\alpha^2-4q^{D+1}) = 4(q^D +1)(q^{D-1}-1)(q^{D-2}-1) > 0$. Therefore,
$$
  (\alpha q^{D-2} -2)^2>q^{2D-4}(\alpha^2-4q^{D+1}).
$$
Furthermore, since $\alpha \ge 2$ and $q > 1$, we have that $\alpha q^{D-2}-2>0$ and $\alpha^2-4q^{D+1} > 0$. Therefore, the above inequality implies
$$
\alpha q^{D-2} - 2 > q^{D-2} \sqrt{\alpha^2-4q^{D+1}},
$$
and the claim follows. 

If $c_2=1$, then the above claim together with Lemma \ref{lem:c2=1} implies that $s^* > q^{-2D-1}$, contradicting Proposition \ref{prop:sandq}(ii). This shows that $c_2 \ge 2$.
\end{proof}

\begin{theorem}
	\label{thm:q>1a}
	Let $\G$ be a bipartite $Q$-polynomial distance-regular graph with diameter $D$ and valency $k \ge 3$. If $D \ge 6$, then $c_2 \ge 2$. In particular, the girth of $\G$ equals $4$.
\end{theorem}

\begin{proof}
	If $\G$ is the $D$-cube or the folded $2D$-cube, then $c_2=2$ and the result follows. Assume now that $\G$ is not the $D$-cube or the folded $2D$-cube. Let $s^*,q$ be scalars as in Lemma \ref{lem:caug}. By Proposition \ref{prop:sandq}(i), we have $q > 1$. The result now follows from Proposition \ref{prop:q>1}.
\end{proof}


\section{$\G$ is bipartite: case $D = 5$}
\label{sec:bipDeq5}

In this section, we consider a bipartite $Q$-polynomial distance-regular graphs with diameter $D = 5$ and valency $k \ge 3$. Our main result is that such graphs have $c_2 \ge 2$.

\begin{proposition}
	\label{prop:th2}
	Let $\G$ denote a bipartite $Q$-polynomial distance-regular graph with diameter $D = 5$ and valency $k \ge 3$, which is not the $5$-cube or the folded $10$-cube. Let $s^*,q$ be scalars as in Lemma \ref{lem:caug}, and let $\theta_0, \theta_1, \ldots, \theta_5$ denote a $Q$-polynomial ordering of the eigenvalues of $\G$. If $c_2=1$, then 
	$$
	  2 \theta_2 = \frac{q^4+q^3+q+1 \pm \sqrt{(q^2+1)(q^2+q+1)(q^4+q^3-2q^2+q+1)}}{q^2}.
	$$
\end{proposition}
\begin{proof}
	Eliminate $s^*$ in  Lemma \ref{lem:caug}(ii) using Lemma \ref{lem:c2=1}, and then rationalize the denominator of the obtained fraction. 
\end{proof}

\begin{corollary}
	\label{cor:th2}
		Let $\G$ denote a bipartite $Q$-polynomial distance-regular graph with diameter $D = 5$ and valency $k \ge 3$, which is not the $5$-cube or the folded $10$-cube. Let $s^*,q$ be scalars as in Lemma \ref{lem:caug}, and let $\theta_0, \theta_1, \ldots, \theta_5$ denote a $Q$-polynomial ordering of the eigenvalues of $\G$.  Define $\beta=q+1/q$, $t=\beta^2+\beta-2$, and assume that $c_2=1$. Then 
		\begin{equation}
			\label{eq:th2}
			2 \theta_2 = \beta^2+\beta-2 \pm \sqrt{\beta (\beta+1) (\beta^2+\beta-4)} = t \pm \sqrt{t^2-4}.
		\end{equation}
\end{corollary}

\begin{proof}
	Equation \eqref{eq:th2} follows immediately from Proposition \ref{prop:th2} and from the definition of parameter $t$. 
\end{proof}

\begin{lemma}
	\label{lem:rational}
	Let $\G$ denote a bipartite $Q$-polynomial distance-regular graph with diameter $D = 5$ and valency $k \ge 3$, which is not the $5$-cube or the folded $10$-cube. Let $s^*,q$ be scalars as in Lemma \ref{lem:caug}, and let $\theta_0, \theta_1, \ldots, \theta_5$ denote a $Q$-polynomial ordering of the eigenvalues of $\G$. Define $\beta=q+1/q$ and $t=\beta^2+\beta-2$, and assume that $c_2=1$. Then the following (i)--(iii) hold:
	\begin{itemize}
		\item[(i)] $t =(\theta_2^2+1)/\theta_2$;
		\item[(ii)] $\sqrt{4t+9} \in \QQ$;
		\item[(iii)] $\theta_2 \ge 1$.
	\end{itemize}
\end{lemma}

\begin{proof}
(i) This follows directly from \eqref{eq:th2}. 

\noindent
(ii) Since $\beta^2+\beta-(t+2)=0$, we have that 
$$
  \beta = \frac{-1 \pm \sqrt{4t+9}}{2}.
$$
As $\beta \in \QQ$, the result follows.

\noindent
(iii) Note that $t > 0$ as $|\beta| > 2$, and so \eqref{eq:th2} implies that $\theta_2  > 0$. The result follows since $\theta_2$ is an integer.

\end{proof}

\noindent
The following result is a folklore and therefore its proof is left to the reader.

\begin{lemma}
	\label{lem:folk}
	Let $m, n$ be positive integers, such that $\sqrt{m/n} \in \QQ$. Then there exist integers $a,b$ with $\gcd(a,b)=1$ such that $m=d a^2$ and $n=d b^2$, where $d = \gcd(m,n)$.
\end{lemma}

\begin{proposition}
	\label{prop:notsquare}
	Let $u$ be a non-negative integer. Then $4u^2+9u+4$ is a perfect square if and only if $u=0$.
\end{proposition}

\begin{proof}
	If $u=0$ then it is clear that $4u^2+9u+4$ is a perfect square. Assume now that $4u^2+9u+4 = \alpha^2$ for some integer $\alpha$. Note that $4u^2+9u+4 = \alpha^2$ is equivalent to the equality
	$$
	(8u+9)^2 - (4 \alpha)^2 = 17.
	$$
	It is a well known fact that every odd prime number $p$ can be written in a unique way as a difference of two squares, namely
	$$
	p = \Big( \frac{p+1}{2} \Big)^2 - \Big( \frac{p-1}{2} \Big)^2,
	$$
	see for example \cite[p. 270]{Burton}. Therefore $8u+9=9$, and so $u=0$.  
\end{proof}

\noindent
We are now ready to prove the main result of this section. 

\begin{theorem}
	\label{thm:t}
	Let $\G$ denote a bipartite $Q$-polynomial distance-regular graph with diameter $D = 5$ and valency $k \ge 3$. Then $c_2 \ge 2$. In particular, the girth of $\G$ is $4$.
\end{theorem}
	
\begin{proof}
	If $\G$ is the $5$-cube or the folded $10$-cube, then $c_2=2$, and the result follows. Assume now that $\G$ is not the $5$-cube or the folded $10$-cube. Let $s^*,q$ be scalars as in Lemma \ref{lem:caug}, and let $\theta_0, \theta_1, \ldots, \theta_5$ denote a $Q$-polynomial ordering of the eigenvalues of $\G$.  Define $\beta=q+1/q$ and $t=\beta^2+\beta-2$, and assume that $c_2=1$. By Lemma \ref{lem:rational}(i) we have $t= (\theta_2^2+1)/\theta_2$, and so $4t+9=(4\theta_2^2+9\theta_2+4)/\theta_2$.  Recall that $\sqrt{4t+9} \in \QQ$ by Lemma \ref{lem:rational}(ii). Therefore, by  Lemma \ref{lem:folk}, there exist integers $a,b$ with $\gcd(a,b)=1$, such that $4\theta_2^2+9\theta_2+4=d a^2$ and $\theta_2=d b^2$, where $d=\gcd(4\theta_2^2+9\theta_2+4,\theta_2)$. It is easy to see that $d \in \{1,2,4\}$. 
	
	Assume first that $d=2$. This implies that $\theta_2$ is even and not divisible by $4$. Recall also that  $\theta_2=2 b^2$, and so $b$ must be odd. We have
	$$
	  4\theta_2^2+9\theta_2+4 = 16 b^4 + 18 b^2 + 4 = 2 a^2,
	$$
	and so
	\begin{equation}
		\label{eq:mod4}
		8 b^4 + 9 b^2 + 2 = a^2.
	\end{equation}
	It follows that $a$ is odd, which implies $a^2 \equiv b^2 \equiv 1 \pmod{4}$. But now \eqref{eq:mod4} yields $11=9+2 \equiv 1 \pmod{4}$, a contradiction.
	
	Assume next that $d=\{1,4\}$, and so Lemma \ref{lem:folk} implies that $4\theta_2^2+9\theta_2+4$ is a perfect square. By Proposition \ref{prop:notsquare} we get $\theta_2=0$, contradicting Lemma \ref{lem:rational}(iii). This concludes the proof.
\end{proof}

\noindent
We are now ready to state and prove the main result of this paper.

\begin{theorem}
	\label{thm:main}
	Let $\G$ denote a bipartite $Q$-polynomial distance-regular graph with diameter $D \ge 3$ and valency $k \ge 3$. Then the girth of $\G$ is equal to $6$ if and only if  $\G$ is  isomorphic either to the Odd graph on a set with cardinality $2D+1$, or to a generalized hexagon of order $(1, k - 1)$.
\end{theorem}
 
\begin{proof}
	 If  $\G$ is  isomorphic either to the Odd graph on a set with cardinality $2D+1$ or to a generalized hexagon of order $(1, k - 1)$, then it is easy to see that the girth of $\G$ is $6$. Now assume that the girth of $\G$ is equal to $6$. By Corollary \ref{cor:girth6}, $\G$ is either almost bipartite or bipartite. The result now follows from Theorem \ref{thm:almost}, Proposition \ref{prop:D=3}, Proposition \ref{prop:D=4}, Theorem \ref{thm:q>1a}, and Theorem \ref{thm:t}.
\end{proof}

\end{document}